\documentclass[12pt, leqno]{amsart}
\usepackage{amssymb, amsmath}
\usepackage{graphicx}
\usepackage{url}
\usepackage{enumerate}
\usepackage{mathabx}

\newtheorem{theorem}{Theorem}[section]
\newtheorem{corollary}[theorem]{Corollary}
\newtheorem{prop}[theorem]{Proposition}
\newtheorem{lemma}[theorem]{Lemma}

\theoremstyle{remark}

\theoremstyle{definition}
\newtheorem{defn}[theorem]{Definition}
\newtheorem{example}[theorem]{Example}

\numberwithin{equation}{section}
\numberwithin{theorem}{section}

\newcommand{\R}{{\mathbb R}}

\providecommand{\abs}[1]{\lvert#1\rvert}
\providecommand{\norm}[1]{\lVert#1\rVert}
\newcommand{\sumin}{\sum_{i=1}^{n}}
\newcommand{\tbar}{\overline{t}}
\newcommand{\ttbar}{\overline{t^2}}
\newcommand{\xbar}{\overline{x}}
\newcommand{\ybar}{\overline{y}}
\newcommand{\xxbar}{\overline{x^2}}
\newcommand{\yybar}{\overline{y^2}}
\newcommand{\xybar}{\overline{xy}}
\newcommand{\var}{{\rm var}}
\newcommand{\cov}{{\rm cov}}

\begin{document}
\thanks{University of the Fraser Valley student Emily Ell implemented the algorithms
in Maple and produced one of the examples.  Emily was funded by a Work/study grant
from this institution.}
\subjclass{Primary 62J05. Secondary 26B5.}
\keywords{least squares, linear regression, completing the square, variance, covariance}
\date{Preprint November 9, 2020.}
%\title[]
\title{Variations on least squares}
\author{Erik Talvila}
\address{Department of Mathematics \& Statistics\\
University of the Fraser Valley\\
Abbotsford, BC Canada V2S 7M8}
\email{Erik.Talvila@ufv.ca}

\begin{abstract}
Three methods of least squares are examined for fitting a line to points in the plane.  Two well known
methods are to minimize sums of squares of vertical or horizontal distances to the line.  Less known is
to minimize sums of squares of distances to the line.  Concise proofs are given for each method using a
combination of the first derivative test for functions of two variables and completing the square.  The
three methods are compared and the distances to the line method appears to be favourable in most 
circumstances.  They generally draw different regression lines.  The method of vertical 
displacements typically gives a slope of too small magnitude while the method of horizontal displacements
typically gives a slope of too large magnitude.  An inequality involving the three slopes is proved.  Rotating all the data points in the same way with
these two methods does not result in the regression line being rotated the same way.  However, the distance
to the line method is invariant under rotations.
\end{abstract}

\maketitle

\section{Introduction}\label{sectionintroduction}
If you had a collection of ordered pairs in the plane and you thought there might
be something approaching a linear relation between them you could use the method
of least squares for linear regression.  Take a look at Figure~\ref{figurethreemethods}.
\begin{figure}[h]
{\includegraphics
[scale=.3]
%[width=6in]
{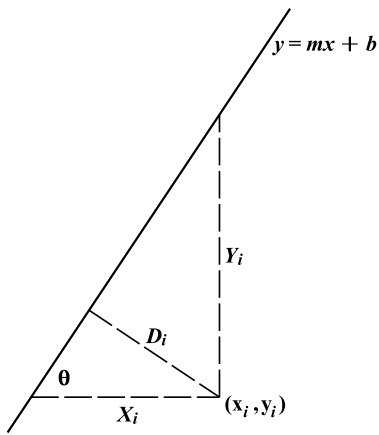}}
\caption{Regression line $y=mx+b$  and data point $(x_i,y_i)$}
\label{figurethreemethods}
\end{figure}
  This shows just
one point $(x_i,y_i)$ of a collection of $n$ points.  The regression line is written $y=mx+b$.
Three measures of how close $(x_i,y_i)$ is to the regression line are the vertical
distance $Y_i$, the horizontal distance $X_i$ and the distance to the line $D_i$.
In the $Y$ method of least squares, one sums up the squares $Y_i^2$ and then chooses
the real numbers $m$ and $b$ so as to minimize this sum.
Similarly, in the $X$ method
with sums of $X_i^2$.  However, in this case it is easiest to write the regression line
as $x=\mu y +\beta$ and then the result follows from the $Y$ method by interchanging
the $x$ and $y$ coordinates.  With the $D$ method one uses sums of squares $D_i^2$.  Here
there is symmetry between the $x$ and $y$ coordinates and it
is more convenient to write the line as $x\sin\theta-y\cos\theta=c$ where $\theta$
is the angle the line makes with the horizontal.  

The
$Y$ method is very well known.  Gauss had worked it out by $1795$ but left it in unpublished
notebooks until finally publishing it in $1809$.  Meanwhile, Legendre wrote a paper on it in
$1805$.  This led to a priority dispute.  For this early history see \cite{merriman},
\cite{merrimanB}, \cite{plackett}
and \cite{stigler}.  
For Gauss's results, translated from Latin and with commentary, see \cite{gauss}.

This has given mathematicians about $200$ years to produce an enormous literature on
least squares.  We will only scratch the surface here, our goal being to compare the three
methods.  We'll find that in general they will draw three different regression lines.
If $m$, $m_X=1/\mu$ and $\tan\theta$ are the slopes of the regression line for the respective 
$Y$, $X$ and $D$ methods, then
we will see that $\abs{m}\leq\abs{m_X}$.  There is equality if and only if all the
data points lie on a line that is neither vertical nor horizontal.
In many cases, $\abs{m}\leq\abs{\tan\theta}\leq\abs{m_X}$.
See Corollary~\ref{corollarymin}.  You can look ahead to Figure~\ref{figureparallel} to see
the discrepancy between the three methods.

The $Y$ method works well for lines that are nearly horizontal, the $X$ method
works well for lines that are nearly vertical, and the $D$ method works best for lines
in between.  All three methods are invariant under translation, meaning that if all
the data points are translated the same way in the plane the regression line translates
that way as well.  However, if all the data points are rotated by the same angle about
some fixed point in the plane, only in the $D$ method is the regression line rotated in
this way.  See Proposition~\ref{proptranslaterotate}.  
This seems like a major failing of the $Y$ and $X$ methods.

Another peculiarity is that if we had data points on two parallel lines, symmetrically
placed in the sense that a point $(x_0,y_0)$ is on one line if and only if there is a point 
$(x_0',
y_0')$ on the other line such that the line segment through $(x_0, y_0)$ and $(x_0',y_0')$ is 
perpendicular to the two parallel lines.   (Rather like the endpoints of rungs on a ladder.)
You might expect the regression line would be midway between the two parallel lines.
This is only so for the $Y$ method if the lines are horizontal and only so for the 
$X$ method if the lines are vertical.  We get the expected midpoint line in the $D$
method whenever the points aren't clustered too close together (Example~\ref{exampleparallellines},
Figure~\ref{figureparallel}).

If all the data points are evenly spaced on a circle then each method fails in its own way,
although the type of failure in the $D$ method is most satisfying.
See Example~\ref{examplecircle}.

The $Y$ method appears in pretty much every textbook on statistics.  For a detailed
discussion see \cite{drapersmith}.
The $X$ method follows from the $Y$ method by interchanging the $x$ and $y$ variables.  
See, for example, \cite[p.~227-229]{spiegelstatistics}.  
These are typically obtained as a calculus problem
in minimizing a function of two variables using the second derivative test \cite{seeley} or 
by completing the square \cite{wetherill}.
We find combining
these two methods gives an easy proof.

The $D$ method is less well known.  Adcock gave a formula for the slope in 1878 \cite{adcock}.
See also \cite{adcockD}, \cite{adcockC}.   This formula is also given on the web page
\cite{weisstein}.  (This web page references a  Mathematica notebook by D.~Sardelis and T.~Valahas,
which is also quoted in \cite{komzsik}.)
The above authors give the regression line in the form $y=mx+b$;
$m$ and $b$ are each obtained by solving a quadratic equation.  It is not clear which roots of
the quadratic to choose.
But the $D$ method is symmetrical with respect to $x$ and $y$. 
Writing the regression line as
$x\sin\theta-y\cos\theta=c$ lets us give a complete algorithm for determining $\theta$
and $c$ and removes
confusion about choosing roots of a quadratic.  It is also easy
to see that this method is invariant under rotation.  In the literature, the $D$ method is
sometimes referred to as ``perpendicular offsets".  An excellent informal introduction to the $D$
method is the web page \cite{johansen}, where the author also notices the rotational
invariance problem.

Our approach is naive in that we do not assume any results from the general theory
of data fitting.  We begin just with the minimization problem in Figure~\ref{figurethreemethods}.
The paper is self contained; only minimization using derivatives of functions
of two variables and some basic linear algebra are needed.

As soon as the data is even roughly linear the $D$ method gives a better fit than the other
two methods.  It's derivation is not overly difficult and the resulting formulas are not
too complicated.  It is our hope that this paper encourages people to use this method in the
future.
\section{Some useful background results}
This is not a paper in statistics but some quantities commonly used in linear regression
arise naturally.  Thus, we'll introduce variance and covariance and prove a few results
that will be useful later.
\begin{defn}
Let $x=(x_1,x_2,\ldots,x_n)$ and $y=(y_1,y_2,\ldots,y_n)$ be vectors in $\R^n$,
where $n\geq 2$.
The {\it mean} of $x$ is $\xbar=(1/n)\sumin x_i$ and $\xybar=(1/n)\sumin x_iy_i$.
The {\it covariance} of
$x$ and $y$ is $\cov(x,y)=\xybar-\xbar\,\ybar$.
The {\it variance} of $x$ is $\var(x)=\cov(x,x)=\xxbar-\xbar^2$.
\end{defn}
The mean is of course the average of the numbers $x_1,x_2,\ldots,x_n$.  The 
variance is a measure of how far these numbers differ from their average.
The variance
can be written as $\var(x)=(1/n)\sumin(x_i-\xbar)^2$.  It is then the mean
square distance from the average.  These quantities can be written in terms of the
norm $\norm{x}=((1/n)\sumin x_i^2)^{1/2}$ but it is more convenient to use the 
square of the norm.
A familiar quantity is the standard deviation, which is the square root of the variance.
The covariance can also be written as $\cov(x,y)=(1/n)\sumin(x_i-\xbar)(y_i-\ybar)$.
The variance and covariance turn out to be natural parameters for these problems
and we will write out answers in terms of them whenever possible.

\begin{lemma}\label{lemmavarcov}
Let $x,y\in\R^n$.  Then\\
(a) $\cov(x,y)=(1/n)\sumin(x_i-\xbar)(y_i-\ybar)=(1/(2n^2))\sum_{1\leq i,j\leq n}(x_i-x_j)(y_i-y_j)=(1/n^2)\sum_{1\leq i<j\leq n}(x_i-x_j)(y_i-y_j)$,\\
(b) $\var(x)=(1/n^2)\sum_{1\leq i<j\leq n}(x_i-x_j)^2$,\\
(c) $\var(x)\geq 0$ and $\var(x)=0$ if and only $x_1=x_2=\ldots=x_n$,\\
(d) $\var(x)=0$ if and only if  $\cov(x,y)=0$ for all $y\in\R^n$.
\end{lemma}
\begin{proof}
(a) The first equality is elementary.  The rest of the proof is by induction on $n$.  Let $M_n=\cov(x,y)$.  Note that $M_1=x_1y_1-x_1y_1=0$
and the other formulas clearly give $0$.
Suppose the formula holds for $n\geq 2$.  Then
\begin{eqnarray*}
n^2M_n & = & (n-1)\sum_{i=1}^{n-1}x_iy_i+\sum_{i=1}^{n-1}x_iy_i + nx_ny_n-\left(\sum_{i=1}^{n-1}x_i+x_n
\right)\left(\sum_{i=1}^{n-1}y_i+y_n
\right)\\
 & = & (n-1)^2M_{n-1} + \sum_{i=1}^{n-1}x_iy_i - \sum_{i=1}^{n-1}x_iy_n -\sum_{i=1}^{n-1}x_ny_i
+(n-1)x_ny_n\\
 & = & \sum_{1\leq i<j\leq n-1}(x_i-x_j)(y_i-y_j) + \sum_{i=1}^{n-1}(x_i-x_n)(y_i-y_n)\\
 & = & \sum_{1\leq i<j\leq n}(x_i-x_j)(y_i-y_j).
\end{eqnarray*}

(b) Put $y=x$ in part (a).

(c) This follows from part (b).

(d) If $x_i=k$ for all $1\leq i\leq n$ then
$$
\cov(x,y)=\xybar-\xbar\,\ybar=\frac{1}{n}\sumin ky_i-\frac{1}{n^2}\sumin k\sumin y_i=k\,\ybar-k\,\ybar=0.
$$
If $\cov(x,y)=0$ for all $y\in\R^n$ then
$$
0=\frac{1}{n}\sumin x_iy_i-\frac{1}{n^2}\sumin x_i\sumin y_i=
\frac{1}{n}\sumin (x_i-\xbar)y_i.
$$
Since this vanishes for all $y_i$ we then have $x_i=\xbar$ for each $i$.
\end{proof}

\section{Three varieties of least squares}\label{sectionregressionlines}
To find the distance from point $(x_i,y_i)$ to the line $ax+by=c$ suppose point $(x_0,y_0)$ is
on the line.  Let the distance to the line be $D$.  A unit normal to the line is $\hat{n}=(a/\sqrt{a^2+b^2},
b/\sqrt{a^2+b^2})$.  The vector $D\hat{n}-(x_i-x_0,y_i-y_0)$ is parallel to the line.  Therefore,
$(D\hat{n}-(x_i-x_0,y_i-y_0))\cdot\hat{n}=0$.  Solving for $D$ gives the well known formula
$D=\abs{ax_i+by_i-c}/\sqrt{a^2+b^2}$.

We will find it is more convenient to minimize the mean square distance.
\begin{theorem}\label{theoremthreeleast}
Let $x,y\in\R^n$.\\
(a) Let $y=mx+b$ be a line.  Define $Y_i=\abs{mx_i+b-y_i}$.  Let $Y(m,b)=(1/n)\sumin Y_i^2$.
Suppose $\var(x)\not=0$.  Then $Y$ has a unique minimum given by
\begin{eqnarray*}
m & = & \frac{\cov(x,y)}{\var(x)}=\frac{n\sumin x_iy_i-\sumin x_i \sumin y_i}{n\sumin x_i^2
-\left(\sumin x_i\right)^2}\\
b & = & \frac{\xxbar\,\ybar-\xbar\,\xybar}{\var(x)}=\frac{\ybar\var(x)-\xbar\cov(x,y)}{\var(x)}=
\frac{\sumin x_i^2\sumin y_i-\sumin x_i\sumin x_iy_i}{n\sumin x_i^2
-\left(\sumin x_i\right)^2}.
\end{eqnarray*}
(b) Let $x=\mu y+\beta$ be a line.  Define $X_i=\abs{\mu y_i+\beta-x_i}$.  
Let $X(\mu,\beta)=(1/n)\sumin X_i^2$.
Suppose $\var(y)\not=0$.  Then $X$ has a unique minimum given by
\begin{eqnarray*}
\mu & = & \frac{\cov(x,y)}{\var(y)}=\frac{n\sumin x_iy_i-\sumin x_i \sumin y_i}{n\sumin y_i^2
-\left(\sumin y_i\right)^2}\\
\beta & = & \frac{\yybar\,\xbar-\ybar\,\xybar}{\var(y)}=\frac{\xbar\var(y)-\ybar\cov(x,y)}{\var(y)}=
\frac{\sumin y_i^2\sumin x_i-\sumin y_i\sumin x_iy_i}{n\sumin y_i^2
-\left(\sumin y_i\right)^2}.
\end{eqnarray*}
(c) Let a line be given by $x\sin\theta-y\cos\theta=c$ where $-\pi/2<\theta\leq\pi/2$.  
Define $D_i=\abs{x_i\sin\theta-y_i\cos\theta-c}$.  
Let $D(\theta,c)=(1/n)\sumin D_i^2$.
The minimum of $D$ occurs at $c=\xbar\sin\theta-\ybar\cos\theta$. 
If $\var(x)=\var(y)$ and $\cov(x,y)=0$ then all values of $\theta$ give the same minimum.
If $\var(x)\not=\var(y)$
then $D$ has a unique minimum given by
\begin{equation}
\tan(2\theta)=\frac{2(\cov(x,y))}{\var(x)-\var(y)}
=\frac{2\left(n\sumin x_iy_i-\sumin x_i \sumin y_i\right)}{n\sumin x_i^2
-\left(\sumin x_i\right)^2 - n\sumin y_i^2
+\left(\sumin y_i\right)^2}.\label{tan2theta}
\end{equation}
Let $E=2\cov(x,y)/[\var(x)-\var(y)]$.
If $\var(x)>\var(y)$ then
$\theta=(1/2)\arctan(E)$.
If $\var(x)<\var(y)$ then
$\theta=(1/2)\arctan(E)+\pi/2$ if $\cov(x,y)\geq 0$ and if $\cov(x,y)<0$ then
$\theta=(1/2)\arctan(E)-\pi/2$.
If $\var(x)=\var(y)$ and $\cov(x,y)>0$ then
$\theta=\pi/4$; if $\cov(x,y)<0$ then
$\theta=-\pi/4$.
\end{theorem}
From Lemma~\ref{lemmavarcov}(c),
if $\var(x)=0$ then all the points are on a vertical line and the method in (a) cannot
be used.
In (b) we have set things up so that vertical lines are possible.  But
the method in (b) cannot be used if $\var(y)=0$ since all the points are then on a
horizontal line (Lemma~\ref{lemmavarcov}(c)).   In (c) the 
line is written in a symmetrical manner that favours neither the $x$ nor $y$ directions.
The line makes an angle $\theta$ with respect to the positive $x$ axis.  We therefore take
$-\pi/2\leq\theta\leq\pi/2$ and identify $\pm\pi/2$ as giving the same slope.  There are no
solvability conditions as with $X$ or $Y$ minimization.

Part (a), and sometimes part (b), appears in multivariable
calculus texts as an example of using partial derivatives and the second derivative test to
find an extremum.  
However, we think it's better to find the minimum by completing the square.
With
quadratics in one variable there is an algorithm for completing the square and this then
indicates the extreme value.  Quadratics
in two variables can be written as the sum or difference of squares of linear combinations of the two
variables but it is not obvious what these linear combinations are.  Using partial derivatives
shows how to do this.  This method has the advantage that it immediately
gives the location and value of the minimum.
We include a proof since the technique is used in the proof of part (c) and
also since the value of the minimum is used in Corollary~\ref{corollarymin}.
The $Y$ minimization of (a) is
in every statistics text that has a section on linear regression.
The $X$ minimization of (b) occurs, for example, in \cite[p.~227-229]{spiegelstatistics}.
Another method of proof is to find the eigenvectors of the quadratic form that appears in
the functions $X$ and $Y$.

Formula \eqref{tan2theta} appears in \cite{johansen}, and in \cite{creasy} under various statistical assumptions.
\begin{proof}
(a) Note that
\begin{eqnarray*}
Y(m,b) & = & \frac{1}{n}\sumin\left(m^2x_i^2+b^2+y_i^2+2mbx_i-2mx_iy_i-2by_i\right)\\
 & = & m^2\xxbar + b^2 +\yybar +2mb\xbar-2m\xybar-2b\ybar.
\end{eqnarray*}
The condition for an extremum is that $Y_m=Y_b=0$.  Notice that
$Y_b(m,b)=2(b+m\xbar-\ybar)=0$ when $b=\ybar-m\xbar$.  Here is the clue for writing
this quadratic form as a sum of squares;
force the term $(m\xbar+b-\ybar)^2$ to be one of the summands.  Indeed, we have
\begin{eqnarray}
Y(m,b)
 & = & m^2\var(x)+\var(y)-2m\,\cov(x,y)+(m\xbar+b-\ybar)^2\notag\\
 & = & \var(x)\left[m-\frac{\cov(x,y)}{\var(x)}\right]^2 
+\frac{\var(x)\var(y)-\cov^2(x,y)}{\var(x)}\notag\\
 & & \quad+(m\xbar+b-\ybar)^2.\label{Ysquare}
\end{eqnarray}
The unique minimum of $Y$ then occurs when $m=\cov(x,y)/\var(x)$ and 
$b=(\xxbar\,\ybar-\xbar\,\xybar)/\var(x)$.

(b) This follows from (a) upon relabelling variables; $(x_i,y_i)\mapsto(y_i,x_i)$,
$m\mapsto\mu$, $b\mapsto\beta$.

(c)  We have
\begin{eqnarray*}
D(\theta,c) & = &  \frac{1}{n}\sumin\left(x_i^2\sin^2\theta + y_i^2\cos^2\theta+c^2-x_iy_i\sin(2\theta)
 -2cx_i\sin\theta+2cy_i\cos\theta\right)\\
 & = &  \xxbar\sin^2\theta +\yybar\cos^2\theta +c^2-\xybar\sin(2\theta) -2c\xbar\sin\theta+2c\ybar\cos\theta.
\end{eqnarray*}
This is not a quadratic form but we can still find its minimum by
using a partial derivative.
We have, $D_c(\theta,c)=2(c-(\xbar\sin\theta-\ybar\cos\theta))=0$ when $c=\xbar\sin\theta-\ybar\cos\theta$.
Force the square of this term to appear and we get
\begin{align*}
&D(\theta,c)\\   
&=\left[c-(\xbar\sin\theta-\ybar\cos\theta)\right]^2+\var(x)\sin^2\theta+\var(y)\cos^2\theta
-\cov(x,y)\sin(2\theta)\\
& =\left[c-(\xbar\sin\theta-\ybar\cos\theta)\right]^2+\frac{1}{2}\left(\var(x)+\var(y)\right)
-\frac{1}{2}\left(\var(x)-\var(y)\right)\cos(2\theta)\\
&\quad-\cov(x,y)\sin(2\theta).
\end{align*}
Let $h(\theta)= -(\var(x)-\var(y))\cos(2\theta)
-2\cov(x,y)\sin(2\theta)$
then $D$ is minimized when $\theta$ is chosen to minimize
$h$
and $c$ is chosen to be $\xbar\sin\theta-\ybar\cos\theta$.
Notice that 
\begin{eqnarray*}
h'(\theta) & = & 2\left[(\var(x)-\var(y))\sin(2\theta)
-2\cov(x,y)\cos(2\theta)\right]\\
 & = & 2\cos(2\theta)\left[(\var(x)-\var(y))\tan(2\theta)
-2\cov(x,y)\right].
\end{eqnarray*}
Since the tangent function is increasing on each interval of its domain, $h'$ is increasing at its zero when $\cos(2\theta)$ and
$\var(x)-\var(y)$ have the same sign.  This point is then a minimum of $h$.
This leads to several cases.

Let $E=2\cov(x,y)/[\var(x)-\var(y)]$.

If $\var(x)>\var(y)$ then $h$ has a minimum at $\abs{\theta}<\pi/4$ such that $\tan(2\theta)=E$.
Hence, $\theta=(1/2)\arctan(E)$.

If $\var(x)<\var(y)$ then $h$ has a minimum at $\pi/4<\theta\leq\pi/2$ or $-\pi/2<\theta<-\pi/4$
such that $\tan(2\theta)=E$.
Hence, $\theta=(1/2)\arctan(E)+\pi/2$ if $E\leq 0$.  If $E>0$ then
$\theta=(1/2)\arctan(E)-\pi/2$.  Note that if $E=0$ ($\cov(x,y)=0$) then $\theta=\pi/2$ since
$-\pi/2<\theta\leq\pi/2$.

If $\var(x)=\var(y)$ then $h'(\theta)=-4\cov(x,y)\cos(2\theta)$.  If $\cov(x,y)>0$ then
$h$ has a minimum at $\theta=\pi/4$.  If $\cov(x,y)<0$ then
$h$ has a minimum at $\theta=-\pi/4$.

If $\var(x)=\var(y)$ and $\cov(x,y)=0$ then all values of $\theta$ give the same minimum,
$D(\theta,\xbar\sin\theta-\ybar\cos\theta)=(\var(x)+\var(y))/2$. 
\end{proof}
In part (c) the regression line is written in terms of $\sin\theta$ and $\cos\theta$.  These
can be computed directly from $\tan(2\theta)$ without using $\arctan$ as in \eqref{tan2theta}.
As in the proof let $E=2\cov(x,y)/[\var(x)-\var(y)]$.  The identities
$\cos^2\theta=[1+\cos(2\theta)]/2$, $\sin^2\theta=[1-\cos(2\theta)]/2$, 
$\cos^2(2\theta)=1/[1+\tan^2(2\theta)]$ let us solve for $\cos\theta$ and $\sin\theta$ in terms
of $\tan(2\theta)$ and hence in terms of $E$.
There are four different formulas for $\cos\theta$ and $\sin\theta$, depending on the signs of
$\cos\theta$, $\sin\theta$ and $\cos(2\theta)$.  The various cases at the
end of the proof show how to make the sign choices.   For example, if $\var(x)>\var(y)$ then $\abs{\theta}<\pi/4$
so $\cos\theta>0$, $\cos(2\theta)>0$  and
$$
\cos\theta={\sqrt \frac{\sqrt{1+E^2}+1}{2\sqrt{1+E^2}}}.$$
If also $\cov(x,y)>0$ then $0<\theta<\pi/4$ and $\sin\theta>0$ so that
$$
\sin\theta={\sqrt \frac{\sqrt{1+E^2}-1}{2\sqrt{1+E^2}}}.$$
Alternatively, if $\cov(x,y)<0$ then $-\pi/4<\theta<0$ and $\sin\theta<0$ so that
$$
\sin\theta=-{\sqrt \frac{\sqrt{1+E^2}-1}{2\sqrt{1+E^2}}}.$$
Other cases are handled similarly.  The results are summarized in the following corollary.

\begin{corollary}\label{corollaryD}
Let $E=2\cov(x,y)/[\var(x)-\var(y)]$.  The values of $\theta$, $\cos\theta$ and $\sin\theta$
for the regression line $x\sin\theta-y\cos\theta=c$ are given as follows.\\
(i) If $\var(x)>\var(y)$ and $\cov(x,y)\geq 0$ then $E\geq 0$ and $0\leq\theta<\pi/4$.  Then  $\theta=(1/2)\arctan(E)$ and
$$
\cos\theta={\sqrt \frac{1+E^2+\sqrt{1+E^2}}{2(1+E^2)}},\quad
\sin\theta={\sqrt \frac{1+E^2-\sqrt{1+E^2}}{2(1+E^2)}}.
$$
If $E\not=0$ then
$
\tan\theta=(-1+\sqrt{1+E^2})/E$.\\
(ii) If $\var(x)>\var(y)$ and $\cov(x,y)\leq 0$ then $E\leq 0$ and $-\pi/4<\theta\leq 0$.  Then $\theta=(1/2)\arctan(E)$ and
$$
\cos\theta={\sqrt \frac{1+E^2+\sqrt{1+E^2}}{2(1+E^2)}},\quad
\sin\theta=-{\sqrt \frac{1+E^2-\sqrt{1+E^2}}{2(1+E^2)}}.
$$
If $E\not=0$ then
$
\tan\theta=(-1+\sqrt{1+E^2})/E$.\\
(iii) If $\var(x)<\var(y)$ and $\cov(x,y)\geq 0$ then $E\leq 0$ and  $\pi/4<\theta\leq\pi/2$. Then $\theta=(1/2)\arctan(E)+\pi/2$ and
$$
\cos\theta={\sqrt \frac{1+E^2-\sqrt{1+E^2}}{2(1+E^2)}},\quad
\sin\theta={\sqrt \frac{1+E^2+\sqrt{1+E^2}}{2(1+E^2)}}.
$$
If $E\not=0$ then
$
\tan\theta=(-1-\sqrt{1+E^2})/E$.\\
(iv) If $\var(x)<\var(y)$ and $\cov(x,y)\leq 0$ then $E\geq 0$ and  $-\pi/2<\theta<-\pi/4$. Then $\theta=(1/2)\arctan(E)-\pi/2$ and
$$
\cos\theta={\sqrt \frac{1+E^2-\sqrt{1+E^2}}{2(1+E^2)}},\quad
\sin\theta=-{\sqrt \frac{1+E^2+\sqrt{1+E^2}}{2(1+E^2)}}.
$$
If $E\not=0$ then
$
\tan\theta=(-1-\sqrt{1+E^2})/E$.\\
(v) If $\var(x)=\var(y)$ and $\cov(x,y)>0$ then $\theta=\pi/4$ and $\cos\theta=\sin\theta=1/\sqrt{2}$ 
and $\tan\theta=1$.\\
(vi) If $\var(x)=\var(y)$ and $\cov(x,y)<0$ then $\theta=-\pi/4$ and $\cos\theta=-\sin\theta=1/\sqrt{2}$
and $\tan\theta=-1$.
\end{corollary}
In the cases where $E<0$ we need to note that $\sqrt{E^2}=-E$ when writing the expression for $\tan\theta$.
Case (v) can be obtained from case (i) by letting $E\to\infty$ or from case (iii) by letting $E\to-\infty$.
Case (vi) can be obtained from case (ii) by letting $E\to-\infty$ or from case (iv) by letting $E\to\infty$.

The conditions appearing at the end of part (c) are independent.  For example, if $n=3$, let $x=(1,0,0)$,
 then
$\var(x)=2/9$.  
If $y=(0,1/\sqrt{3},-1/\sqrt{3})$ then $\var(y)=2/9$ and $\cov(x,y)=0$.
If $y=(0,1,0)$ then $\var(y)=2/9$ and $\cov(x,y)=-1/9<0$.
If $y=(1,0,1)$ then $\var(y)=2/9$ and $\cov(x,y)=1/9>0$.

The next corollary points out various relationships between the slopes of the regression
lines.  In particular, the $Y$ minimization method tends to underestimate the magnitude
of the slope while the $X$ minimization method tends to overestimate it.  
Example~\ref{exampleparallellines} also shows this phenomenon.
\begin{corollary}\label{corollarymin}
(a) In the $Y$ method the regression line can be written $(y-\ybar)\var(x)=(x-\xbar)\cov(x,y)$;
in the $X$ method it can be written $(x-\xbar)\var(y)=(y-\ybar)\cov(x,y)$; 
in the $D$ method it can be written $(x-\xbar)\sin\theta=(y-\ybar)\cos\theta$.
(b) The point $(\xbar,\ybar)$ is on the regression line of each method.  Hence, the three lines 
intersect here.
(c) For all $x,y\in\R^n$ we have $\var(x)\var(y)\geq\cov^2(x,y)$ with equality
if and only if all the points $(x_i,y_i)$ lie on the same line.
(d)  If they exist, the slopes of the regression lines in the $Y$, $X$ and $D$ methods have
the same sign.
(e) Let $m=\cov(x,y)/\var(x)$ be the slope of the regression line in the $Y$ method and let
$m_X=\var(y)/\cov(x,y)$ be the slope in the $X$ method.
The inequality
$\abs{m}\leq\sqrt{\var(y)/\var(x)}\leq \abs{m_X}$ holds, with equality if and only if
all the points $(x_i,y_i)$ lie on the same line.
(f) Let $\tan\theta$ be the slope of the regression line in the $D$ method.
The inequality $\abs{m}\leq\abs{\tan\theta}\leq \abs{m_X}$ holds in cases (i), (ii), (v) and (vi) of
Corollary~\ref{corollaryD}.  The inequality also holds in cases (iii) and (iv) if
$2\cov^2(x,y)\geq\var(x)\abs{\var(x)-\var(y)}$.
\end{corollary}
\begin{proof}
(a) This follows from the formulas for $m$, $b$, $\mu$, $\beta$ and $c$ in the theorem.
(b) This now follows immediately from (a) of this corollary.
(c) In the proof of the theorem, 
the quantity $Y(m,b)\geq 0$ and is written as a sum of squares in \eqref{Ysquare}.
Since $\cov(x,y)\geq 0$ (Lemma~\ref{lemmavarcov}(c)) the inequality follows.
There is equality  if and only if the sum of squares of $Y_i$ is zero.  But this happens
if and only if $mx_i+b-y_i=0$ for all $1\leq i\leq n$, i.e., the points are on the
regression line.  If the line is vertical  then $\var(x)=0$ (Lemma~\ref{lemmavarcov}(d)) and the
inequality follows trivially.
(d) Since the variances are positive the slopes $m$ and $m_X$ 
have the same sign.  From Corollary~\ref{corollaryD} this can be seen to be the same 
sign as $\tan\theta$. (e)  The inequality follows from (c).
(f) In cases (v) and (vi) of Corollary~\ref{corollaryD} the inequality follows from part (c) above.
 Using the expression for $\tan\theta$ in Corollary~\ref{corollaryD}(i) and
$E=2\cov(x,y)/[\var(x)-\var(y)]$ we get
$$
\tan\theta=\frac{\sqrt{[\var(x)-\var(y)]^2+4\cov^2(x,y)}-[\var(x)-\var(y)]}{2\cov(x,y)}.
$$
And, since the slopes are positive and $m=\cov(x,y)/\var(x)$ then $\tan\theta\geq m$ if and only if
\begin{equation}
\var(x)\sqrt{[\var(x)-\var(y)]^2+4\cov^2(x,y)}\geq 2\cov^2(x,y)+\var(x)[\var(x)-\var(y)],\label{varinequality}
\end{equation}
where we have used the fact that the covariance is positive in this case.  Since the right hand side
of the above inequality is non-negative we get an equivalent inequality by squaring both sides.  This
then simplifies to $\var(x)\var(y)\geq\cov^2(x,y)$ (part (c)).  The inequality with $m_X$ is similar.
So is the case (ii) in Corollary~\ref{corollaryD}, being careful to note that the slopes are negative,
and the inequality is reversed upon multiplication by $\cov(x,y)$.  In case (iii) of 
Corollary~\ref{corollaryD} 
the slopes are positive and $[\var(x)-\var(y)]<0$ so $\tan\theta\geq m$ if and only if \eqref{varinequality}
holds.  But now, we get an equivalent inequality by squaring only when the right hand side of \eqref{varinequality}
 is non-negative.
The other cases are similar.
\end{proof}
Part (c) is equivalent to the Cauchy--Schwarz inequality.

When $\var(x)$ is zero the slope is not defined in Theorem~\ref{theoremthreeleast}(a).  However,
writing the regression line as $(y-\ybar)\var(x)=(x-\xbar)\cov(x,y)$ doesn't help since then
the covariance is also zero (Lemma~\ref{lemmavarcov}(d)).  Similarly if $\var(y)=0$ for the
$X$ minimization method.  There are no such restrictions for the $D$ method.

When we write down the minimum of 
$X(\mu,\beta)$ and $D(\theta,c)$ we see we get the same inequality as in part (c) of the corollary.
For $D$ the four cases in Corollary~\ref{corollaryD} must be used.

\section{Comparison of the three methods}
Now we show why minimizing the distance to the line (Theorem~\ref{theoremthreeleast}(c))
is often superior to the vertical and horizontal methods.  

You might expect the regression line to have certain invariances under rigid transformations.
We will say it is invariant under translations if, when you translate all points $(x_i,y_i)$ the
same way the regression line translates the same way as well.  All three methods have this type
of invariance.  We will say the regression line is invariant under rotations if, 
when all points $(x_i,y_i)$ are rotated the
regression line undergoes the same rotation.  The regression line is invariant under rotations in
the distance to the line method but not in the other two methods.  That's strange.  You rotate all
the data points by the same angle but the regression line does not rotate like that.   

Suppose $u,v\in\R$ and the translation is $(x_i,y_i)\mapsto (x_i',y_i')=(x_i+u,y_i+v)$
for each $1\leq i\leq n$.  If there is invariance under translation and the regression line is
$y=mx+b$ then it becomes $y'-v=m(x'-u)+b$ or $y'=mx'+b'$ where $b'=b-mu+v$.  The
slope does not change under translation.  If the line is written $x= \mu y+\beta$ then the translated
version is $x'=\mu y'+\beta'$ where $\beta'=\beta-\mu v +u$.  For $x\sin\theta-y\cos\theta=c$
the invariant regression line becomes $x'\sin\theta-y'\cos\theta=c'$ where $c'=c+u\sin\theta-v\cos\theta$.

Suppose a rotation about the origin is given by 
\begin{equation}
\left(
\begin{array}{c}
x_i'\\
y_i'
\end{array}
\right)
=
\left(
\begin{array}{cc}
\cos\phi & -\sin\phi\\
\sin\phi & \cos\phi
\end{array}
\right)\!
\left(
\begin{array}{c}
x_i\\
y_i
\end{array}
\right).\label{rotation}
\end{equation}
The inverse transformation is obtained by changing the sign of $\phi$.
If the line $y=mx+b$ is invariant then after rotation it becomes 
$(\cos\phi-m\sin\phi)y'=(\sin\phi+m\cos\phi)x'+b$.
If the line $x=\mu y+\beta$ is invariant then after rotation it becomes 
$(\cos\phi+\mu\sin\phi)x'=(-\sin\phi+\mu\cos\phi)y'+\beta$.
If the line $(x-\xbar)\sin\theta=(y-\ybar)\cos\theta$ is invariant then after rotation it becomes 
$(x'-\xbar')\sin(\theta')=(y'-\ybar')\cos(\theta')$ where $\theta'=\theta+\phi$.
Because of translation invariance it doesn't matter which point we rotate
about.

We will see that the three methods in general give different regression lines for the same data and that the
 $Y$ minimization
method favours making the magnitude of the slope too small and the $X$ minimization method favours making
the magnitude of the slope too large, whereas the distance to the line minimization method gets things
(nearly) right.
\begin{prop}\label{proptranslaterotate}
(a) The regression lines in the $Y$, $X$ and $D$ minimization methods of Theorem~\ref{theoremthreeleast} are invariant under
translations.
(b) If all the points $(x_i,y_i)$ lie on a line that is not vertical then the regression line for $Y$
minimization is this line.  If all the points lie on a line that is not horizontal then the regression line 
for $X$
minimization is this line.
If all the points lie on a line then the regression line for $D$
minimization is this line.
(c) Only the regression line for $D$ minimization is invariant under rotation.
\end{prop}
\begin{proof}
(a) We can see from Lemma~\ref{lemmavarcov}(a), (b) that the variance and covariance are invariant under translation.
But in Theorem~\ref{theoremthreeleast} the slopes are written strictly in terms of the variance and
covariance.  Hence, they do not change.  And, $b$, $\beta$ and $c$ change as in the paragraph above.

(b) Due to translation invariance we can assume the line is through the origin.  Accordingly,
let the points be $(x_i,\alpha x_i)$ for $\alpha\in\R$.  For $Y$ minimization we have
$$
m=\frac{\cov(x,y)}{\var(x)}=\frac{n\alpha\sumin x_i^2-\alpha\left(\sumin x_i\right)^2}{n\sumin x_i^2-\left(\sumin x_i\right)^2}=\alpha
$$
and
$$
b  =  \frac{\xxbar\,\ybar-\xbar\,\xybar}{\var(x)}=\frac{\alpha\sumin x_i^2\sumin x_i
-\alpha\sumin x_i\sumin x_i^2}{n\sumin x_i^2-\left(\sumin x_i\right)^2}=0.
$$
This gives the regression line $y=\alpha x$ and all data points lie on this line.
The $X$ minimization is similar if we write the points as $(\alpha x_i,x_i)$. We have $\ybar=\alpha\xbar$, $\var(y)=\alpha^2\var(x)$ and
$\cov(x,y)=\alpha\var(x)$.  For $D$ minimization we then have
$E=2\cov(x,y)/(\var(x)-\var(y))=2\alpha/(1-\alpha^2)$.  We get case (i) in Corollary~\ref{corollaryD} if
$0\leq\alpha<1$.  Then $1+E^2=[(1+\alpha^2)/(1-\alpha^2)]^2$ so that after a little algebra we get
$$
\sin\theta={\sqrt \frac{1+E^2-\sqrt{1+E^2}}{2(1+E^2)}}=\frac{\alpha}{\sqrt{1+\alpha^2}}.
$$
Similarly, $\cos\theta=1/\sqrt{1+\alpha^2}$ so that $\tan\theta=\alpha$.  And, $c=\xbar\sin\theta-\ybar\cos\theta=0$.  The regression
line is then $y=\alpha x$ which is the line on which the data lie.  The other cases in 
Corollary~\ref{corollaryD} are similar, being careful to watch the sign of $\alpha$.

(c) Let $(x_1,y_1)=(0,0)$, $(x_2,y_2)=(1,0)$ and $(x_3,y_3)=(2,1)$.  A calculation gives
$\xbar=1$, $\xxbar=5/3$, $\var(x)=2/3$, $\ybar=1/3=\yybar$, $\var(y)=2/9$, $\cov(x,y)=1/3$.  This shows the $Y$ minimization
regression line is $y=x/2-1/6$ and the $X$ minimization regression line is 
$x=(3/2)y +1/2$.  If we rotate with $\phi=\pi/2$ we get the new set of points
$(x_1',y_1')=(0,0)$, $(x_2',y_2')=(0,1)$, $(x_3',y_3')=(-1,2)$.
Now, $\overline{x'}=-1/3$, $\overline{x'^2}=1/3$, $\var(x')=2/9$,
$\overline{y'}=1$, $\overline{y'^2}=5/3$, $\var(y')=2/3$, $\cov(x',y')=-1/3$.  These give the $Y$ minimization line $y'=-(3/2)x'+1/2$ and the
$X$ minimization line $x'=-y'/2+1/6$.  Whereas, according to the formulas preceding the proposition, if there
is invariance under rotation, we should get
$y'=-2x'+1/3$ and $x'=-(2/3)y'+1/3$.

Now show the $D$ minimization method is invariant under rotation.  Using \eqref{rotation} we have
$\xbar=\overline{x'}\cos\phi+\overline{y'}\sin\phi$ and $\ybar=-\overline{x'}\sin\phi+\overline{y'}\cos\phi$.  These give
\begin{eqnarray*}
\var(x) & = & \var(x')\cos^2\phi+\cov(x',y')\sin(2\phi)+\var(y')\sin^2\phi,\\
\var(y) & = & \var(x')\sin^2\phi-\cov(x',y')\sin(2\phi)+\var(y')\cos^2\phi,\\
2\cov(x,y) & = & -\var(x')\sin(2\phi)+2\cov(x',y')\cos(2\phi)+\var(y')\sin(2\phi).
\end{eqnarray*}
The quantity $E$ and its counterpart $E'$ in the primed variables are related by
\begin{eqnarray*}
E & = & \frac{2\cov(x,y)}{\var(x)-\var(y)}=\tan(2\theta)\\
 & = & \frac{-[\var(x')-\var(y')]\sin(2\phi)+2\cov(x',y')\cos(2\phi)}{
[\var(x')-\var(y')]\cos(2\phi)+2\cov(x',y')\sin(2\phi)}\\
 & = & \frac{E'-\tan(2\phi)}{1+E'\tan(2\phi)}. 
\end{eqnarray*}
Solving for $E'$ and using the tangent addition formula gives
$$
E'=\frac{E+\tan(2\phi)}{1-E\tan(2\phi)}=\frac{\tan(2\theta)+\tan(2\phi)}{1-\tan(2\theta)\tan(2\phi)}
=\tan([2(\theta+\phi)]).
$$
Hence, the new angle for the regression line is $\theta'=\theta +\phi$.  Also,
\begin{align*}
&(x-\xbar)\sin\theta-(y-\ybar)\cos\theta\\
&=[(x'-\xbar')\cos\phi+(y'-\ybar')\sin\phi]\sin\theta-[-(x'-\xbar')\sin\phi+(y'-\ybar')\cos\theta]\cos\theta\\
&=(x'-\xbar')\sin(\theta+\phi)-(y'-\ybar')\cos(\theta+\phi).
\end{align*}
Hence, the line is invariant under rotation.
\end{proof}

Of course, we can also see that the $Y$ and $X$ minimization methods are not invariant under rotation
by rotating onto a vertical or horizontal line for which the respective regression lines do not exist.

For the above example we have $\var(x)=2/3$, $\var(y)=2/9$, $\cov(x,y)=1/3$ so $E=3/2$.
Using part (i) of Corollary~\ref{corollaryD} we see that for the points $(0,0)$, $(1,0)$ and $(2,1)$ of
part (c) the regression line has
$$
\tan\theta=\frac{\sqrt{13}-2}{3}=\tan(0.5\arctan(1.5))\doteq
0.53518.
$$  The line can be written as $y=x\tan\theta-\xbar\tan\theta+\ybar$.  It is then approximately
$y=(0.53518)x-0.20185$.
The three regression lines are different.  The inequality in Corollary~\ref{corollarymin}(e)
reads $1/2<1/\sqrt{3}<2/3$ ($m<\sqrt{\var(y)/\var(x)}<m_X$).
The inequality in Corollary~\ref{corollarymin}(f)
reads $1/2<0.53518<2/3$ ($m<\tan\theta<m_X$). 

Now we look at how the three methods perform in two test cases.  In the first, the data points 
are symmetrically arrayed on two parallel lines.  In the second the points are on a circle.
In both cases the $D$ method outperforms the $X$ and $Y$ methods.

\begin{example}\label{exampleparallellines}  
Suppose we had all the data
points arrayed symmetrically on two parallel lines.  By this we mean there are parallel lines
$L_+$ and $L_-$ such that point $(x_0,y_0)$ is on line $L_-$ if and only if there is a point $(x_0',
y_0')$ on $L_+$ so that the line through $(x_0, y_0)$ and $(x_0',y_0')$ is perpendicular to $L_+$
and $L_-$.  It seems reasonable that the regression line would be the unique line $L$ parallel
to $L_+$ and $L_-$ lying midway between them.  This only happens for the $D$ minimization method
and only in certain cases.
To see this, note that due to translation invariance we can assume $L$ is through
the origin.  First consider the case of vertical lines.  Let $A>0$ and let $L_+$ be the line
$x=A$ and $L_-$ the line $x=-A$.  We take a vector $t\in\R^n$ with $\var(t)>0$ and define
$(x_{\pm i},y_{\pm i})
=(\pm A,t_i)$ for $1\leq i\leq n$.  Then $x,y\in\R^{2n}$.  We have 
$\xbar=(1/(2n))\sumin(x_{+i}+x_{-i})=(1/(2n))\sumin(A-A)=0$ and $\xxbar=(1/(2n))\sumin
2A^2=A^2$.  Hence, $\var(x)=A^2$.  Similarly, $\ybar=\tbar$, $\yybar=\ttbar$ and
$\var(y)=\var(t)$.  This gives $\cov(x,y)=0$.  Using the formulas in Theorem~\ref{theoremthreeleast},
the $Y$ minimization regression line is $y=\bar{t}$ (not good) and the $X$ minimization line is $x=0$
(good, we get line $L$).  What about the $D$ minimization method?  Since the covariance is zero we get $E=0$ and
 the angle depends on
the relative size of $\var(x)$ and $\var(y)$.
If $A^2>\var(t)$ then $\theta=(1/2)\arctan(0)=0$ and $c=\xbar\sin\theta-\ybar \cos\theta=-\bar{t}$.
The regression line is $y=\bar{t}$.
If $A^2<\var(t)$ then $\theta=(1/2)\arctan(0)+\pi/2=\pi/2$ and $c=\xbar$.
The regression line is $x=0$.  Does this make sense?  Yes!  If the points $t_i$ are clustered
close together ($\var(t)$ small) and the lines are far apart ($A$ large) then the regression line
is horizontal.  From far away the data look like two fuzzy clusters at $(\pm A, \bar{t})$.
If these conditions are reversed then we get the vertical line $L$.  The data are spread out along
lines $L_+$ and $L_-$ and the line in between gets drawn.  But, the $Y$ minimization method always
draws the horizontal line $y=\bar{t}$ and the $X$ minimization method always draws the line $x=0$,
no matter the distribution of points on $L_+$ and $L_-$.  The transition case is $A^2=\var(t)$
and then $\var(x)=\var(y)$
and $\cov(x,y)=0$ so the $D$ method gives every line through the origin.

To see what happens when $L_\pm$ are not vertical let
$M$ and $B>0$ be real numbers.  Line $L_+$ is $y=Mx+B$ and line $L_-$ is $y=Mx-B$.

Because of rotation invariance in the $D$ method (Proposition~\ref{proptranslaterotate}(c)) 
we get the same effect here as with vertical lines. Notice in the above case that $A$ is the distance from each
line $L_{\pm}$ to the origin.  Using the formula at the beginning of Section~\ref{sectionregressionlines}
the distance from $L_\pm$ to the origin is $B/\sqrt{M^2+1}$.  The regression line is then $L$ if and only if
$B^2/(M^2+1)<\var(t)$.

For the other cases, note that a
vector parallel to $L_\pm$ is $(1,M)$.  Let $t\in\R^n$.  Define $(x_i,y_i)=(t_i,Mt_i+B)$ be
a point on $L_+$ for $1\leq i\leq n$.
For
$(x_{-i},y_{-i})$ to be the corresponding point on $L_-$ then,
the symmetry condition is $(x_{-i}-t_i,M(x_{-i}-t_i)-2B)\cdot(1,M)=0$.  This gives
$x_{-i}=t_i+2MB/(1+M^2)$ and $y_{-i}=Mt_i+(M^2-1)B/(M^2+1)$.  A calculation now yields
\begin{eqnarray*}
\xbar & = & \tbar+\frac{MB}{M^2+1}\\
\xxbar & = & \ttbar +\frac{2MB\tbar}{M^2+1}+\frac{2M^2B^2}{(M^2+1)^2}\\
\var(x) & = & \var(t)+\frac{M^2B^2}{(M^2+1)^2}\\
\ybar & = & M\tbar+\frac{M^2B}{M^2+1}\\
\yybar & = & M^2\ttbar+\frac{2M^3B\tbar}{M^2+1}+\frac{(M^4+1)B^2}{(M^2+1)^2}\\
\var(y) & = & M^2\var(t)+\frac{B^2}{(M^2+1)^2}\\
\xybar & = & M\ttbar +\frac{2M^2B\tbar}{M^2+1}+\frac{M(M^2-1)B^2}{(M^2+1)^2}\\
\cov(x,y) & = & M\var(t)-\frac{MB^2}{(M^2+1)^2}.
\end{eqnarray*}

For the $Y$ minimization method we now have 
$$
m=\frac{\cov(x,y)}{\var(x)}=\frac{M\var(t)-MB^2/(M^2+1)^2}{\var(t)+M^2B^2/(M^2+1)^2}
$$
and this equals $M$ if and only if $MB=0$.  But $B>0$ so the slope of the regression
line is $M$ if and only if $M=0$.  If $M=0$ we see that
$$
b=\frac{\ybar\var(x)-\xbar\cov(x,y)}{\var(x)}=\frac{0-0}{\var(t)}=0.
$$
The regression line is then the $x$ axis, similar to the case of vertical lines and the
$X$ method.

For the $X$ minimization method,
$$
\frac{1}{\mu}  =  \frac{\var(y)}{\cov(x,y)}=\frac{M^2\var(t)+B^2/(M^2+1)^2}{M\var(t)-MB^2/(M^2+1)^2}.
$$
This equals $M$ if and only if $B=0$.  Hence, the $X$ method never has a regression line with the 
same slope as the parallel lines.
\end{example}

Figure~\ref{figureparallel} shows this phenomenon.  The dotted line is the $Y$ method, the dashed
line is the $X$ method and the solid line is the $D$ method.
Notice the lines intersect in a point as in
Corollary~\ref{corollarymin}(b) and the slopes are arranged as in (f) of this corollary.
In the $D$ method the regression line
bisects the parallel lines of data points, while in the $Y$ method the slope is too small and in the
$X$ method the slope is too large.  
In this example, $M=2$ and $B=40$.  
\begin{figure}[h]
{\includegraphics
[scale=.3]
%[width=6in]
{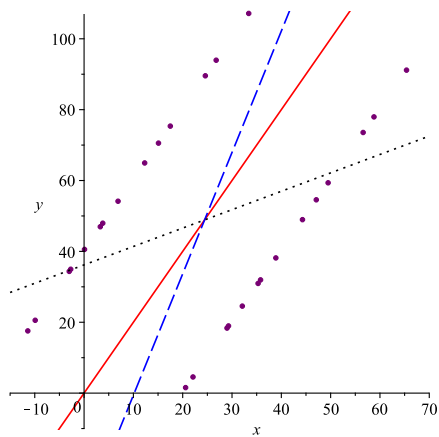}}
\caption{Symmetric data on parallel lines}
\label{figureparallel}
\end{figure}

\begin{example}\label{examplecircle}
Now consider all data points on a circle.  Due to translation invariance we can assume the
centre is the origin.  We will take it to be the unit circle.  It's not hard to show that
changing the radius doesn't change the results.  Let $(x_i,y_i)=(\cos(2\pi i/n),\sin(2\pi i/n))$
for $1\leq i\leq n$ and assume $n\geq 3$.  
This collection of points is symmetric with respect to reflections across
the $x$ axis (achieved by the transformation $i\mapsto n-i$).  There is symmetry across the $y$ axis 
if and only if $n$ is even.  Using arithmetic modulo $n$ this is achieved with the transformation
$i\mapsto n/2-i$.  At the end of the example these points are rotated arbitrary amounts, breaking
any $x$ and $y$ symmetry.

Calculate the quantities in Theorem~\ref{theoremthreeleast}.
The following formulas for sums of sines and cosines are well known:
$$
\sumin \sin(i\theta)=\frac{\sin[(n+1)\theta/2]\sin(n\theta/2)}{\sin(\theta/2)}
\qquad
\sumin \cos(i\theta)=\frac{\sin[(n+1/2)\theta]}{2\sin(\theta/2)}-\frac{1}{2}.
$$
We have
$$
\xbar  =  \frac{1}{n}\sumin\cos\left(\frac{2\pi i}{n}\right)
  =  \frac{\sin[2\pi(n+1/2)/n]}{2n\sin(\pi/n)}-\frac{1}{2n}=\frac{\sin(\pi/n)}{2n\sin(\pi/n)}-\frac{1}{2n}=0.
$$
And, using the cosine double angle formula,
\begin{eqnarray*}
\var(x) & = & \xxbar -0  =  \frac{1}{n}\sumin\cos^2\left(\frac{2\pi i}{n}\right) =\frac{1}{2n}\sumin\left[1+\cos\left(
\frac{4\pi i}{n}\right)\right]\\
 & = & \frac{1}{2}+\frac{1}{2n}\left(\frac{\sin[4\pi(n+1/2)/n]}{2\sin(2\pi/n)}-\frac{1}{2}\right)=\frac{1}{2}.
\end{eqnarray*}
Similarly,
$$
\ybar  =  \frac{1}{n}\sumin\sin\left(\frac{2\pi i}{n}\right)
  =  \frac{\sin[\pi(n+1)/n]\sin(\pi)}{n\sin(\pi/n)}=0.
$$
Then,
\begin{eqnarray*}
\var(y) & = & \yybar -0  =  \frac{1}{n}\sumin\sin^2\left(\frac{2\pi i}{n}\right) =\frac{1}{2n}\sumin\left[1-\cos\left(
\frac{4\pi i}{n}\right)\right]\\
 & = &\frac{1}{2}- \frac{1}{2n}\left(\frac{\sin[4\pi(n+1/2)/n]}{2\sin(2\pi/n)}-\frac{1}{2}\right)=\frac{1}{2}.
\end{eqnarray*}
And,
\begin{eqnarray*}
\cov(x,y) & = & \frac{1}{n}\sumin\cos\left(\frac{2\pi i}{n}\right)\sin\left(\frac{2\pi i}{n}\right)
=\frac{1}{2n}\sumin\sin\left(\frac{4\pi i}{n}\right)\\
 & = & \frac{\sin[2\pi(n+1)/n]\sin(2\pi)}{2n\sin(2\pi/n)}=0.
\end{eqnarray*}
Now we can find the regression lines using formulae from Theorem~\ref{theoremthreeleast}.  
For the $Y$ minimization method  we get $m=b=0$; the $x$ axis for each $n\geq 3$.
For the $X$ minimization method  we get $\mu=\beta=0$; the $y$ axis for each $n\geq 3$.
For the $D$ minimization method we have
$\xbar=\ybar=0$ so $c=0$. 
Since $\var(x)=\var(y)$ and $\cov(x,y)=0$ the angle $\theta$ is arbitrary.  Hence, we get every
line through the origin.  Somehow this seems more satisfying than breaking the symmetry by choosing a
coordinate axis for the regression line.

It is remarkable that if we take $0<\alpha<2\pi$ and use the points 
$(\cos(\alpha +2\pi i/n),\sin(\alpha+2\pi i/n))$ we get the same results.  Because of the
sine and cosine addition formulas, $\cos(\alpha +2\pi i/n)$ and $\sin(\alpha+2\pi i/n)$ are linear combinations
of $\cos(2\pi i/n)$ and $\sin(2\pi i/n)$ so we get the same values for $\xbar$, $\xxbar$, $\ybar$, $\yybar$ and
$\cov(x,y)$ as above.  Only if $\alpha$ is a multiple of $\pi/n$ is there now symmetry
across the $x$ or $y$ axes but the $Y$ method still gives a horizontal line, the $X$ method
still gives a vertical line and the $D$ method still gives every line through the origin.
\end{example}

\end{document}